\relax
\documentclass[letterpaper,twocolumn]{article} 
\usepackage{times}  
\usepackage{helvet} 
\usepackage{courier}  
\usepackage[hyphens]{url}  
\usepackage{graphicx} 
\usepackage{graphicx}  
\usepackage{tikz}
\usepackage{amsmath}
\usepackage{amssymb}
\usepackage{amsthm}
\newtheorem{definition}{Definition}
\newtheorem{proposition}{Proposition}
\newtheorem{lemma}{Lemma}
\newtheorem{corollary}{Corollary}
\newtheorem{theorem}{Theorem}

\date{}

\title{Approximate Gradient Descent Convergence Dynamics \\ for Adaptive Control on Heterogeneous Networks}
\author{Jean Carpentier\textsuperscript{\rm 1} and
Sebastien Blandin\textsuperscript{\rm 2}\\ 
\textsuperscript{\rm 1}Ecole Polytechnique, Palaiseau, France\\
\textsuperscript{\rm 2}IBM Research, Singapore\\
jean.carpentier@polytechnique.edu, sblandin@sg.ibm.com 
}
\begin{document}
\maketitle
\begin{abstract}
Adaptive control is a classical control method for complex cyber-physical systems, including transportation networks. In this work, we analyze the convergence properties of such methods on exemplar graphs, both theoretically and numerically. We first illustrate a limitation of the standard backpressure algorithm for scheduling optimization, and prove that a re-scaling of the model state can lead to an improvement in the overall system optimality by a factor of at most $\mathcal{O}(k)$ depending on the network parameters, where $k$ characterizes the network heterogeneity. We exhaustively describe the associated transient and steady-state regimes, and derive convergence properties within this generalized class of backpressure algorithms. Extensive simulations are conducted on both a synthetic network and on a more realistic large-scale network modeled on the Manhattan grid on which theoretical results are verified.
\end{abstract}
\begin{section}{Introduction}
We consider the scheduling problem on queuing networks, and specifically on urban road networks. Concretely, the problem consists of the allocation of time slots to traffic lights at intersections. While the routing policy may impact the stability of the scheduling solution~\cite{Boyer2015}, the routing problem is often considered decoupled~\cite{nikolova2006optimal} and is not addressed here.

The computational complexity of the scheduling problem on large-scale road networks has motivated the search for efficient decentralized algorithms only requiring local knowledge of network properties and efficient in the absence of coordination. Such decentralized approaches have proven quite efficient in practice~\cite{smith2013smart,xie2012schedule}, have connections with fluid dynamic models~\cite{Blandin2015}, and are amenable to agent-based learning methods such as reinforcement learning~\cite{richter2007natural}.

In the context of communication networks, the \textit{backpressure} algorithm~\cite{tassiulas1992stability} provides a throughput-maximizing control policy, i.e. a policy that guarantees that given any feasible flow, the maximal network queue size is asymptotically bounded. Furthermore, the backpressure policy requires only evaluation of queue size on neighboring road links. 

Several properties of the backpressure algorithm make it appealing for adaptive control of dynamical road networks. First, the backpressure solution is a policy, which by definition is able to handle variability of the network state. Second it can be implemented and deployed in a fully decentralized manner since it requires only local information. Lastly, it comes with theoretical guarantees on the queue size, and has been shown to perform very well in practice.

In the context of intelligent transport systems, significant research efforts have been dedicated to extensions of the backpressure algorithm in recent years, in particular to address the specifics of traffic light scheduling on road networks~\cite{varaiya2009universal}. The case of unknown routing rates was also investigated~\cite{gregoire2014back}, as well as the case of queues with finite capacity~\cite{gregoire2015capacity}, see also~\cite{hucoping} and~\cite{Hu2017a}.

In-depth theoretical analysis of the backpressure properties, such as its behavior under heavy load conditions~\cite{Bramson99}, or its link with game theory~\cite{tenbusch2014guaranteeing} have also been investigated~\cite{singh2015maxweight,moharir2013maxweight}. Similar results exist for the closely related max weight algorithm~\cite{Stolyar04} and an adapted backpressure algorithm~\cite{Venkataramanan10} has been devised in this context. The backpressure algorithm was also shown to be a greedy gradient descent over the quadratic potential~\cite{Wunder12}. In this context, acceleration methods have been proposed~\cite{Zargham13}. We refer to~\cite{Moeller10} for application to wireless sensor networks.

One of the most explicit drawbacks of the backpressure algorithm occurs at a network scale: in steady state, for certain model networks, queue sizes strictly decrease from the origin to the destination along every possible path~\cite{Stolyar09},~\cite{Ying11}, meaning that commuters traveling over longer paths incur longer queues. In this body of work, adjustments have been proposed via the consideration of an additional design cost explicitly accounting for path lengths, hence attempting to compensate that drawback. Other limitations have been investigated for specific network configurations~\cite{Stolyar11}.

In this work, we propose to analyze the convergence dynamics of the backpressure algorithm. Specifically, we are interested in a fine-grained analysis of the backpressure algorithm in different regimes, and associated convergence properties. We first explicitly re-cast the backpressure algorithm as a more general approximate gradient descent method, and show that in that class of methods, significant performance gaps exist depending on the choice of parameters. These general conclusions are derived on the case of a fundamental building block for network flow analysis, namely the $2\times1$ network including two upstream links connected to one downstream link.

We then conduct an in-depth theoretical analysis of the $2\times1$ network, and verify these results in simulation. We confirm experimentally that the results from the theoretical analysis obtained on a simple network apply to realistic networks such as the Manhattan grid.

The main contributions of this work include:
\begin{itemize}
\item illustration of arbitrarily large performance gaps in the backpressure class of adaptive control algorithms,
\item theoretical identification and characterization of transient and stationary regimes of the $2 \times 1$ network under backpressure algorithms,
\item numerical validation of theoretical properties and illustration on realistic networks such as the Manhattan grid.
\end{itemize}
The rest of this article is organized as follows. We first introduce notations and formulate the problem considered. We then present our main results on the convergence dynamics on the $2\times1$ network. We subsequently analyze the relative convergence of two instances of approximate gradient descent algorithms. Finally, we present detailed numerical results of the algorithm performance, and conclude.
\end{section}
\begin{section}{Preliminaries}\label{sec:Prelims}
\begin{subsection}{Network model}
We consider a discrete-time network of queues with $q_{l,m}(t) \in \mathbb{R^{+}}$ denoting the (continuous) number of vehicles queuing at location $l$ at time $t$ with the intention of traveling to the downstream link $m$ next. In the transportation context, each $q_{l,m}(\cdot)$ represents a distinct queue of vehicles waiting to cross an intersection with segregated movements (e.g. turn left, go straight, turn right). We also note $q$ the vector of $q_{l,m}$ and omit the time dependency for compactness. Queuing networks can also model public transport and multi-modal networks~\cite{horni2016multi}, although the emergence of mobile data often requires hybrid approaches~\cite{faster2019}.

The outflow $s_{l,m}(t)$ of queue $q_{l,m}(t)$ at time $t$ is the maximum number of vehicles able to cross the intersection within a time slot, defined as the minimum of the queue size and the queue capacity, assumed static:
\begin{equation}
s_{l,m}(t) = \min(q_{l,m}(t),c_{l,m}).
\label{eq:throughput}
\end{equation}
\noindent When this minimum $s_{l,m}(t)$ is reached at the queue size $q_{l,m}(t)$, the intersection is in \textit{unsaturated regime}, and when the minimum is reached at the queue capacity $c_{l,m} \in \mathbb{R^{+}_{*}}$, the intersection is in \textit{saturated regime}. 

Given initial conditions $q_{l,m}(t=0)$ for the queues, and prescribed source and sink flows $e_{l,m}(t)$ specifying the number of vehicles entering and leaving the network, the conservation of vehicles reads:
\begin{align}
q_{l,m}(t+1) = & \, q_{l,m}(t) + r_{l,m}(t)\sum\limits_k u_{k,l}(t) s_{k,l}(t) \nonumber \\
& - u_{l,m}(t)\,s_{l,m}(t) + e_{l,m}(t)
\label{eq:Conserv}
\end{align}
\noindent where $r_{l,m}(t) \in [0,1]$ is the proportion of vehicles reaching node $l$ intending to visit node $m$ next, and such that $\sum_{m}r_{l,m}(t) = 1$, and $u_{l,m}(t) \in \{0,1\}$ is a control variable specifying whether queue $q_{l,m}(t)$ is activated, i.e. has green light, at time $t$. For each intersection, the activation set follows standard constraints encoding compatible movements, (e.g. in the case of left-hand driving, turn right movements can be activated simultaneously, but not go straight and turn left movements).

Given an objective function $V(q)$ satisfying Lyapunov properties, the \textit{scheduling problem} is concerned with the design of an activation policy $u_{l,m}(\cdot)$ with good properties with respect to $V(\cdot)$.
\end{subsection}
\begin{subsection}{Backpressure algorithm}
In this section we recall some existing results. The \textit{backpressure algorithm}~\cite{tassiulas1992stability} provides \textit{maximal throughput stability} in the sense that if the inputs flows are feasible in expectation, then the queue sizes are asymptotically bounded.
\begin{definition}
The \textit{backpressure policy} is the solution $u$ to the maximization problem :
\begin{equation}
\max \limits_{u} \sum\limits_{l,m}{\left( q_{l,m} - \sum\limits_k{q_{m,k} r_{m,k}}  \right) c_{l,m} u_{l,m}}.
\label{eq:backpressure}
\end{equation}
\label{def:Backpressure}
\end{definition}
It can be shown, e.g. see~\cite{Wongpiromsarn2014}, that this objective function leads to activating at each decision point the traffic movement maximizing the difference between its upstream queue and its downstream queue, hence the term ``backpressure''. 

We first show that the back pressure policy, although arising from a local greedy formulation, corresponds to an approximate gradient descent step. 
\begin{proposition}
The backpressure algorithm~\eqref{eq:backpressure} is an approximate gradient descent step update on the objective function $V(q) = \frac{1}{2}\sum\limits_{l,m} q^2_{l,m} = \frac{1}{2}q^{T}q$.
\end{proposition}
\begin{proof}
If we note $\delta(t+1)=q(t+1)-q(t)$, the one-step temporal difference in the objective function reads:
\begin{equation}
	V(q(t+1)) - V(q(t)) = \delta(t+1)^T q(t) +\frac{1}{2}\delta(t+1)^T\delta(t+1).
\label{eq:diff}
\end{equation}
Expanding $\delta(t+1)$ using the conservation equation~\eqref{eq:Conserv} we can re-write $\delta(t+1)^T q(t)$ as: 
\begin{multline*}
	\delta(t+1)^T q(t) = \\ E^T q - \sum\limits_{l,m} \left( q_{l,m} - \sum\limits_k q_{m,k} r_{m,k}  \right) u_{l,m} s_{l,m}
\end{multline*}
where E is the vector of $e_{l,m}$, and the time-dependence is omitted on the right-hand side. In the saturated regime, the first term on the right-hand side of equation~\eqref{eq:diff} dominates, and a steepest gradient descent step on the approximate temporal difference $\delta(t+1)^T q(t)$ reads:
\begin{multline}
\arg\min_u \delta(t+1)^T q(t) = \\ \arg\max_u \sum\limits_{l,m} \left( q_{l,m} - \sum\limits_k q_{m,k} r_{m,k}  \right) u_{l,m} s_{l,m}.
\label{eq:greedyStep}
\end{multline}
\noindent Approximating the throughput~\eqref{eq:throughput} as $s_{l,m} \approx c_{l,m}$ leads to the definition of the backpressure~\eqref{eq:backpressure}, which corresponds to making the approximation that the queues are in the saturated regime.
\end{proof}
\noindent Motivated by expression~\eqref{eq:greedyStep}, in the following we define the \textit{priority} of a queue as:
\begin{equation}
	p_{l,m} = \left( q_{l,m} - \sum\limits_k{q_{m,k} r_{m,k}}  \right) c_{l,m}.
\label{eq:Priority}
\end{equation}
\noindent This view of backpressure as a general one step update for an approximate gradient descent in the context of adaptive control motivates us to consider a generalization of the objective function via re-scaling. Specifically, given $\gamma_{l,m} > 0$, we consider a generalized objective function $V(q) = \frac{1}{2}\sum\limits_{l,m} \gamma_{l,m} q^2_{l,m}$ associated with the generalized priorities:
\begin{equation}
	p_{l,m} = \left( \gamma_{l,m} q_{l,m} - \sum\limits_k{\gamma_{m,k} q_{m,k} r_{m,k}}  \right) c_{l,m}.
	\label{eq:PwithGamma}
\end{equation}
\noindent We now illustrate that this re-scaling can impact the performance of the approximate gradient descent method by an arbitrary factor depending on the network heterogeneity. We focus the analysis on the comparison between two values of $\gamma_{l,m}$, the case of $\gamma_{l,m} = 1$ which corresponds to the classical backpressure, and the case of $\gamma_{l,m} = 1/c_{l,m}$ which corresponds to a variant of the backpressure algorithm where time spent in the queue is the quantity to be optimized (since $q/c$ is the steady-state saturated regime approximation of time spent in the queue).
\end{subsection}
\begin{subsection}{Heterogeneous Flows}\label{sec:limitations}
We now define a simple but fundamental example and show the limitations of the backpressure algorithm on that case. Consider a simple network with $2$ upstream queues $q_{1,3}$, $q_{2,3}$ and $1$ downstream queue $q_{3,4}$, where heterogeneity between the upstream queues is parameterized by a factor $k$. This topology corresponds to the classical \textit{merge} junction in traffic engineering~\cite{daganzo1995ctm}, see~\cite{piccoli} for the underlying mathematical theory of network fluid-dynamics model.

Given a reference capacity $c$, link capacities are defined as $c_{2,3} = k \, c_{1,3} = kc$. Link inflows are defined as $f_{2,3} = k \, f_{1,3} = k\eta c$. Without loss of generality, we also assume that the downstream queue is constant since the main object of this study is the competing dynamics of upstream queues given an arbitrary downstream queues. For stability, we also make the classical assumption that the inflow is lower than the uniform capacity, i.e. $\eta_i = \frac{f_i}{c_i} \leq 0.5$. 
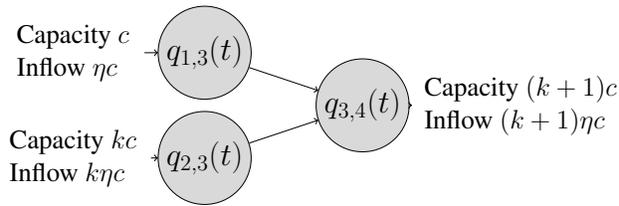
\begin{figure}[htb!]
\centering
    \scalebox{0.7}{\begin{tikzpicture}[
      mycircle/.style={
         circle,
         draw=black,
         fill=gray,
         fill opacity = 0.3,
         text opacity=1,
         inner sep=2pt,
         minimum size=20pt,
         font=\LARGE},
     myhiddencircle/.style={
         circle,
         draw=white,
         align=left,
         font=\Large}]

	\node[myhiddencircle] (hidden_upstream_queue_one) at (-2.5,1) {Capacity $c$\\ Inflow $\eta c$};
	\node[myhiddencircle] (hidden_upstream_queue_two) at (-2.5,-1) {Capacity $kc$\\ Inflow $k \eta c$};
	\node[myhiddencircle] (hidden_downstream_queue) at (6,0) {Capacity $(k+1) c$\\ Inflow $(k+1) \eta c$};
    \node[mycircle] (first_upstream_queue) at (0,1) {$q_{1,3}(t)$};
    \node[mycircle] (second_upstream_queue) at (0,-1) {$q_{2,3}(t)$};
	\node[mycircle] (downstream_queue) at (3,0)        {$q_{3,4}(t)$};

    \draw[->] (first_upstream_queue) edge (downstream_queue) %
              (second_upstream_queue) edge (downstream_queue)
              (hidden_upstream_queue_one) edge (first_upstream_queue)
            (hidden_upstream_queue_two) edge (second_upstream_queue)
            (downstream_queue) edge (hidden_downstream_queue);
\end{tikzpicture}}
\caption{\textbf{$2\times1$ parametric junction} with two upstream queues and one downstream queue.}
\label{fig:ParametricJunction}
\end{figure}

\noindent We now prove that the classical backpressure algorithm~\eqref{eq:backpressure} can perform arbitrary poorly for heterogeneous networks. We present the results for the case $q_{3,4}(t) =0$.
\begin{proposition}
Consider the $2 \times 1$ network from Figure~\ref{fig:ParametricJunction}, with $\eta = 0.5$ and $q_{3,4}(t) = 0$. A basic scheduling alternating activation of each upstream queue leads to: 
\begin{equation}
q_{1,3}(t) \approx \eta c \quad \text{and} \quad q_{2,3}(t) \approx k \eta c
\label{eq:naive}
\end{equation}
while the backpressure activation rule~\eqref{eq:backpressure} yields:
\begin{equation}
q_{1,3}(t) \geq k^2 \eta c \quad \text{and} \quad q_{2,3}(t) \geq k \eta c
\end{equation}
\end{proposition}
\begin{proof}
First, given that in this discrete time setting each vehicle spends at least one time step in the queue, we have $q_{l,m} \geq f_{l,m}$. Here, since the demand is feasible, an alternating schedule would result in $q_{1,3}(t) \approx f_{1,3}$, and similarly for $q_{2,3}$, which proves~\eqref{eq:naive}. 

Second, on this example, the backpressure priorities from equation~\eqref{eq:Priority} read $p_{l,m}(t) = c_{l,m} q_{l,m}(t)$. Since $q_{l,m}(t) \geq f_{l,m}$, we have $p_{2,3}(t) = c_{2,3} q_{2,3}(t) \geq k^2 c^2 \eta$. The queue $q_{1,3}$ is activated only if $p_{1,3}(t) = c_{1,3} q_{1,3}(t) > p_{2,3}(t) \geq k^2 c^2 \eta$, or equivalently $q_{1,3}(t) > k^2 c \eta$. When $q_{1,3}$ goes below that value, queue $q_{2,3}$ is activated because it has higher priority.

Hence the minimal queue size over time under backpressure is $k$ times the queue size under the alternating schedule. A similar analysis can be conducted with non-zero $q_{3,4}$, the result is obtained for a sufficient upstream queue heterogeneity compared to the downstream queue value.
\end{proof}

In the rest of the article, we investigate the convergence dynamics of the approximate gradient descent dynamics under the generalized backpressure priorities~\eqref{eq:PwithGamma}.
\end{subsection}
\end{section}
\begin{section}{Stability Domain}\label{sec:Theory}
We first characterize the asymptotic convergence of the parametric $2\times 1$ network from Figure~\ref{fig:ParametricJunction}. To simplify the notations on this network, we index the upstream queues by $i \in \{1,2\}$.
\begin{subsection}{Regime Types}
For $\eta < 0.5$, the demand is feasible, capacities exceed input flows, which means that at least one queue is activated more often than needed to process the input flow, which means $s_{u,i}(t) < c_i$. For the \emph{$2\times1$} junction, either only one upstream queue is in the unsaturated regime, and we call this network regime \textit{R1}, or both upstream queues are in the unsaturated regime, and we call this network regime $R2$.

Due to space limitation, we focus on the regime $R1$ which exhibits more complex behavior (indeed under the condition $\eta < 0.5$ regime $R2$ can be proven to be a transient regime evolving into $R1$ eventually) and serves as illustrative example of a limitation of the classical backpressure algorithm in the example of previous section. With a slight abuse of notations, we use the saturation state, \textit{s} for saturated and \textit{u} for unsaturated, to index the queue (e.g. $(u,s)=(1,2)$ when queue $1$ is unsaturated and queue $2$ is in the saturated state).
\end{subsection}
\begin{subsection}{A re-scaled backpressure algorithm}
\label{ssn:Monoundersaturation}
A discussed earlier, a queue is bounded below by a $q_{i,min}$ which is the value of its input flow:
\begin{equation}
	q_{i,min} = f_i \leq q_i(t).
	\label{eqn:MinQ}
\end{equation}
\noindent The expressions of the associated minimal priority $p_{i,min}$ can be derived by instantiating equation~\eqref{eq:PwithGamma} with the priority value~\eqref{eqn:MinQ}: 
\begin{equation*}
	p_{i,min} = p_i(q_i = q_{i,min}) \leq p_i(t)
	\label{eqn:MinP}
\end{equation*}
In order for a queue to be activated, its priority must be at least greater than the minimal priority of every competing queue. We call $p_{act}$ this a-priori minimal priority to be activated:
\begin{equation}
	p_{act} = \max\limits_j p_{j,min},
	\label{eqn:Pact}
\end{equation}
\noindent and the expression associated $q_{i,act}$ follows from~\eqref{eq:PwithGamma} as:
\begin{equation}\begin{split}
	q_{i,act} = q_i(p_i = p_{act}) = \frac{\gamma_0}{\gamma_i} Q + \frac{1}{\gamma_i c_i} p_{act}.
	\label{eqn:MinQforall}
\end{split} \end{equation}
In regime $R1$ one queue is saturated and the other queue is unsaturated. It follows from~\eqref{eqn:MinQforall} that:
\begin{equation}
	q_{s,act} = \frac{\gamma_0}{\gamma_s} Q + \frac{\gamma_j f_j - \gamma_0 Q}{\gamma_s c_s}  c_j \geq c_s 
	\label{eqn:CondMonoParams}
\end{equation}
\noindent where $j$ is such that $p_j = p_{act}$. If $j = s$, meaning that queue $j$ is in the saturated state, then equation~\eqref{eqn:CondMonoParams} simplifies to $\eta \geq 1 $, which is impossible by assumption. Hence the queue $j$ must be in the unsaturated state when reaching its activation priority, and the other queue is in the saturated state:
\begin{equation}
	p_{act} = p_{u, min} = p_u(q_u=f_u)
	\label{eqn:PactWithO}
\end{equation}
\noindent We can now express sufficient conditions for each of the queues to be in saturated or unsaturated state, depending on the exogenous network parameters.
\begin{proposition}In the $R1$ regime, the following states can exist:
	\begin{itemize}
		\item $(u,s) = (1,2)$ when the following conditions are satisfied:
			\begin{itemize}
				\item for $\gamma = \mathbf{1}_n$: $ Q \geq \frac{k^2 - \eta}{k-1}c $
				\item for $\gamma = [\frac{1}{c_i}]_{i \in \{1,...,n\}}$: $ Q \geq \frac{k - \eta}{k-1}(k+1)c$
			\end{itemize}
		\item $(u,s) = (2,1)$ when the following conditions are satisfied:
			\begin{itemize}
				\item for $\gamma = \mathbf{1}_n$: $ Q \leq \frac{k^2\eta - 1}{k-1}c $
				\item for $\gamma = [\frac{1}{c_i}]_{i \in \{1,...,n\}}$: $Q \leq \frac{k\eta - 1}{k-1}(k+1)c$
			\end{itemize}
	\end{itemize}
	\label{prop:BothMono}
\end{proposition}
\begin{proof}
The bounds follow from instantiating equation~\eqref{eqn:CondMonoParams} on the cases $(u,s) = (1,2)$ and $(u,s) = (2,1)$.
\end{proof}
\noindent We represent the different phases characterized by Proposition~\ref{prop:BothMono} in Figure~\ref{fig:BothMono}.
\begin{figure}[htb!] 
		\centering
		\includegraphics[width=\columnwidth]{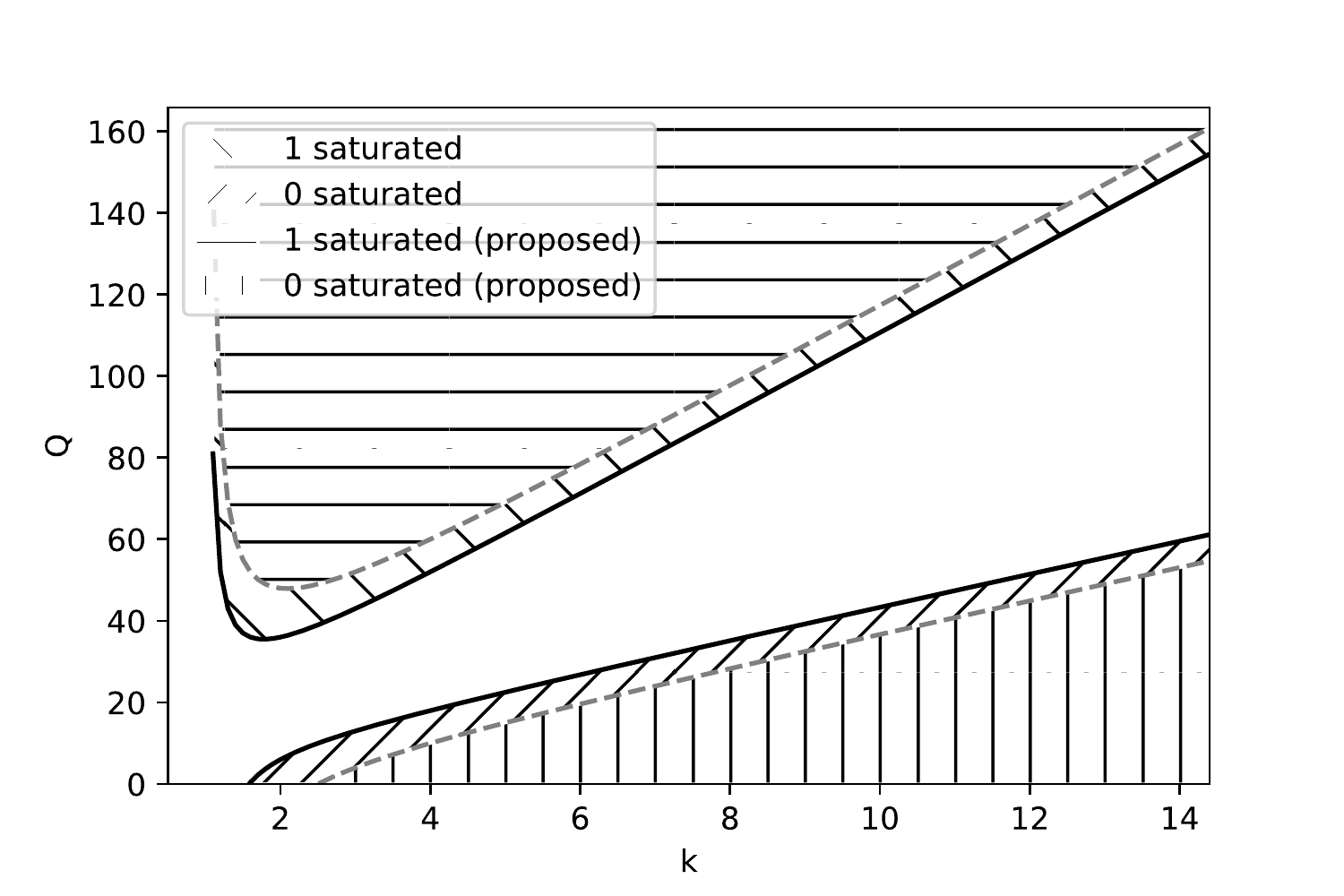}	
		\caption{\textbf{Saturated and unsaturated} states for each link in $R1$ regime. When $k$ is large compared to $Q$ (bottom right part of the chart), i.e. in case of significant heterogeneity, the smaller capacity queue (named $0$) is saturated, while for larger $Q$ values (top left part of the chart), i.e. for high network load and resulting high coupling between high-demand links, the larger capacity queue (named $1$) is saturated.}
		\label{fig:BothMono}
\end{figure}

We observe that with the proposed non-uniform weights $\gamma_{l,m}=1/c_{l,m}$, the saturation region decreases, which suggests a better utilization of the network capacity.
\end{subsection}
\end{section}
\begin{section}{Convergence properties}\label{sec:ConvProp}
In this section we characterize the transient and stationary phases of the system, and provide theoretical results on the impact of the weights $\gamma$ on the limit of the gradient descent.
\begin{subsection}{Transient state}
In regime $R1$, under the assumption that $\eta < 0.5$, we first prove that in general, there exists a transient regime during which the maximal priority decreases. 
\begin{definition}
	The rolling min-max over 2 time slots reads:
	\begin{equation*}
		\tilde{p}_{max}(t) = \min\limits_{v \in \{t-1, t\}} \max\limits_{i \in \{1,2\}} p_i(v)
	\end{equation*}
\end{definition}
\noindent Lemma~\ref{lemma:PmaxDecreases} provides conditions to ensure the overall decrease in queue size during the transient phase.
\begin{lemma}
	For the $2 \times 1$ network, with $\eta < 0.5$, we have:
	\begin{equation}
		\forall t,	\tilde{p}_{max}(t+2) \leq \tilde{p}_{max}(t).
		\label{eqn:PmaxSoft}
	\end{equation}
	
	Furthermore, with $i^* = \arg\max p_i(t)$ and $j$ the other queue, and with $\Delta p(t) = p_{i^*}(t) - p_j(t) \geq 0$, we have $\forall t$:
	\begin{equation}\begin{array}{c}
		q_{i^*}(t) > 2f_{i^*} 
		\text{ and }\Big(q_j(t) > f_j \text{ or } \Delta p(t) > 0\Big)
		\\ \implies
		\tilde{p}_{max}(t+2) < \tilde{p}_{max}(t) 
		\label{eqn:PmaxStrict}
	\end{array}\end{equation}
	\label{lemma:PmaxDecreases}
\end{lemma}
\begin{proof}
The technical proof can be found in the supplemental material.
\end{proof}

We can now use the specificity of the R1 regime to describe the dynamics of the $2 \times 1$ network:

\begin{corollary}
	If $p_{max}(t) = p_{s}$, i.e. $i^*(t) = s$, the saturated queue is activated for the time slot $t$ and:
	\begin{equation*}
		q_u(t) > q_{u,min}\implies\tilde{p}_{max}(t+2) < p_{max}(t)
		\label{eqn:mActivePmaxDec}
	\end{equation*}
	and
	\begin{equation}
		\left\{ \begin{array}{l} q_u(t) = q_{u,min} \\ \Delta p(t) = 0 \end{array} \right.
		\implies 
		\left\{  \begin{array}{l} i^*(t) = s \\ i^*(t+1) = i^*(t+2) = u \end{array} \right.
		\label{eqn:mActiveAtMostOnce}
	\end{equation}
	 \label{corr:PmaxDecreasesStrictly}
	 In other words $q_s$ is activated at most once consecutively and it is followed by two activations of $q_u$.
\end{corollary}
\begin{proof}
	The proof is obtained via a disjunction over the values of $q_u(t)$, details included in the supplemental material.

	
	
\end{proof}
\noindent Corollary~\ref{corr:PmaxDecreasesStrictly} states first that the end of the transient state is related to the time when the unsaturated queue takes its minimal value. Equation~\eqref{eqn:mActiveAtMostOnce} describes the queuing process in steady state: the saturated queue cannot be activated more than once consecutively. This implies that the unsaturated queue remains very close to its minimum value.
\end{subsection}
\begin{subsection}{Characterization of the steady state}
We now generalize the result from Corollary~\ref{corr:PmaxDecreasesStrictly}.

\begin{proposition}The steady state is reached after a finite time $t_0$ and is characterized by:
	\begin{equation*}
		\forall t \geq t_0,\tilde{p}_{max}(t) = p_{act}
	\end{equation*}
	where $p_{act}$ is defined by~\eqref{eqn:PactWithO}.
	\label{prop:PmaxIsPact}
\end{proposition}
\begin{proof}
	First we prove that if queue $s$ is regularly activated then $\tilde{p}_{max}(t)$ strictly decreases overtime. It eventually reaches the minimal value $p_{act}$. 
	
	Lemma~\ref{lemma:PmaxDecreases} states that $\tilde{p}_{max}(t)$ is non-increasing therefore bounded above by $P$. The input flow within one time slot is bounded (actually equal to $f_i$) so between two time slots, $p_{max}(t)$ remains bounded above with $p_{max}(t) \leq P + \max\limits_i f_i$. In the \emph{$2\times1$} network, $p_i$ is affine of $q_i$, so queues are also bounded above. Hence there exists an interval such that every queue is activated at least once within this interval otherwise the constants input flows would make it diverge.
	
	Let us consider such an interval and a time $t_s$ in that interval at which $q_s$ is activated. Equation~\eqref{eqn:PmaxSoft} states that $\tilde{p}_{max}$ is non increasing. Corollary~\ref{corr:PmaxDecreasesStrictly} states that while $q_u(t) > q_{u,min}$, $\tilde{p}_{max}$ is decreasing strictly. Let us consider the other case, for which $q_u(t_s) = q_{u,min}$, corresponding to $p_u(t_s) = p_{act}$ according to~\eqref{eqn:PactWithO}. We have that $p_s(t_s) = p_{max}(t_s)$:
	
	\begin{itemize}
	\item if $p_s(t_s) > p_{act}$, equivalently $\Delta p(t_s) > 0$ and~\eqref{eqn:PmaxStrict} yields $\tilde{p}_{max}(t)$ strictly decreasing,
	\item else, $\Delta p(t_s) = 0$, $p_{max}(t) = p_{act}$, so $\forall t > t_s, \tilde{p}_{max}(t) = p_{act}$ which is the lower bound.
	\end{itemize}
	
	\noindent In any case, either $\tilde{p}_{max}(t)$ decreases or has reached its lower bound after the activation of $q_s$, which means that in finite time, $\tilde{p}_{max}(t)$ is arbitrarily close to $p_{act}$.
\end{proof}
\noindent It follows that $\tilde{p}_{max}(t)$ being constant bounds the size of the queues below and above. Experimentally, these bounds are quite tight, see Figure~\ref{fig:TransientToSteady}.

\begin{theorem}
	In the $R1$ regime, given $\eta < 0.5$, the \emph{$2\times1$} network converges to a steady state where under a generalized backpressure algorithm with $\gamma > 0$:
	\begin{itemize}
		\item for the unsaturated queue $q_u$:
		\begin{equation*}
			f_u \leq q_u(t) \leq 2f_u
		\end{equation*}
		\item and for the saturated queue:
		\begin{equation*}
			q_{s,act} + (f_s - c_s) \leq q_s(t) \leq q_{s,act} + f_s	
		\end{equation*}
		with
		\begin{equation*}
			q_{s,act} =  \frac{\gamma_0}{\gamma_s}(1 - \frac{c_u}{c_s}) Q + \frac{\gamma_u c_u}{\gamma_s c_s} f_u
		\end{equation*}
	\end{itemize}
	\label{theo:Qbounds}
\end{theorem}
\begin{proof}
	The proof consists of expanding the results from Proposition~\ref{prop:PmaxIsPact}.
	
	\textbf{For the unsaturated queue} we obtain:
	\begin{equation}
		p_u(t) \leq \tilde{p}_{max}(t) + \gamma_u f_u c_u = p_{act} + \gamma_u f_u c_u
		\label{eqn:AYGH}
	\end{equation}
	 Equation~\eqref{eqn:PactWithO} states that $p_{act}= p_u(q_u=f_u)$ Consequently the queues associated with the priorities from equation~\eqref{eqn:AYGH} read:
	 \begin{equation*}
	 	q_{u}(t) \leq 2f_u
	\end{equation*}
	and the lower bound corresponds to \eqref{eqn:MinQ}.
		
	\textbf{For the saturated queue}:
	\begin{equation*}\begin{array}{rl}
		p_s(t) &\leq \tilde{p}_{max}(t) + \gamma_s f_s c_s \\
		&\ = p_{act} + \gamma_s f_s c_s \\
		&\ = \underbrace{(\gamma_u f_u - \gamma_0 Q) c_u}_{p_{act}\ (cf. \eqref{eqn:PactWithO})} +  \gamma_s f_s c_s	\\
		&\ = p_{s,max}
	\end{array}\end{equation*}
	
	\noindent We define as well the upper bound $q_{s,max}$ of $q_s$ s.t. $p_{s,max} = p_s(q_s = q_{s,max}) = (\gamma_s q_{s,max} -  \gamma_0 Q)$. Consequently:
	\begin{equation*}\begin{array}{rl}
		q_{s,max} = f_s + \frac{\gamma_0}{\gamma_s}(1 - \frac{c_u}{c_s}) Q + \frac{\gamma_u c_u}{\gamma_s c_s} f_u
	\end{array}\end{equation*}
	\noindent For the minimum $q_{s,min}$, we note that it is reached after an activation done with $p_s(t) = p_{act}$ (defined as the minimal priority to be activated): 		
	\begin{equation*}\begin{array}{rl}
		p_{s,min} &= p_{act} + \gamma_s (f_s - c_s) c_s \\
			&= p_{s,max} - \gamma_s c_s c_s
	\end{array}\end{equation*}
	\noindent This corresponds to associated queue values $q_{s,min} = q_{s,max} - c_s$ and the result follows.
\end{proof}

We illustrate the results of this section using a numerical simulation reported in Figure~\ref{fig:TransientToSteady}.
\begin{figure}[htb!] 
	\centering
	\includegraphics[width=\columnwidth]{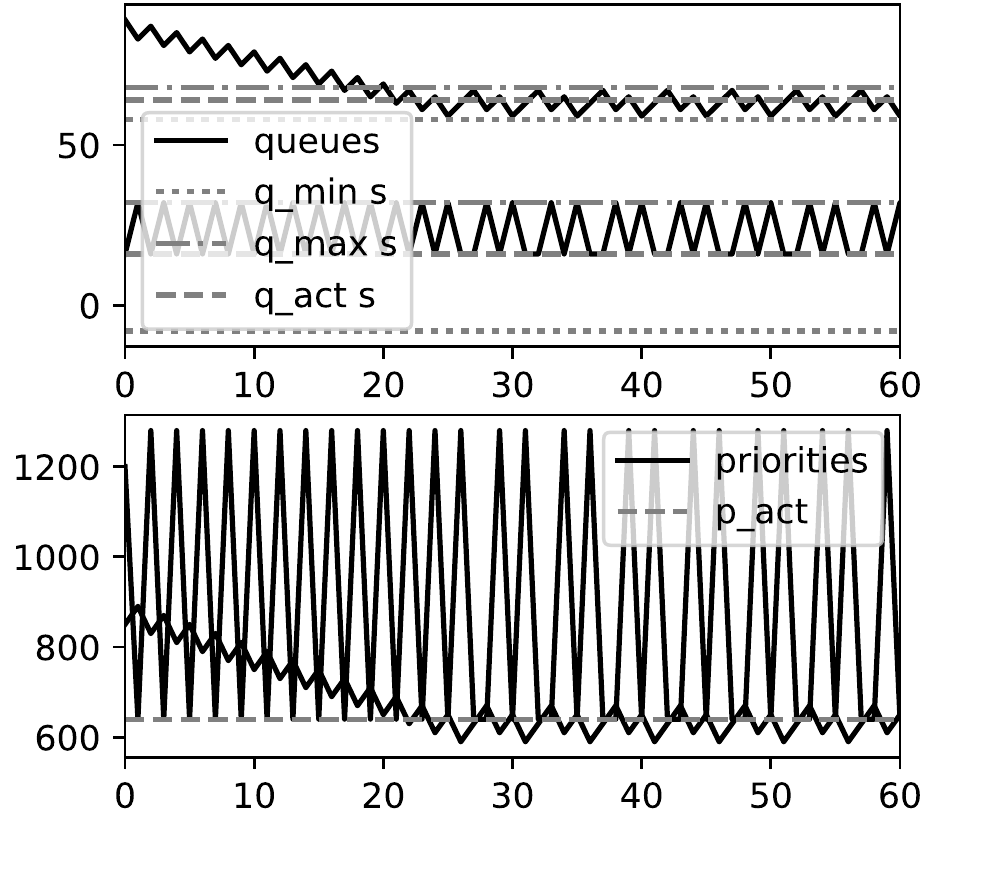}
	\caption{\textbf{Convergence to the steady state} for queues (top) and priorities (bottom). As predicted by the theory, the priority decreases (here until time step $25$) and then oscillates around a constant activation priority.}
	\label{fig:TransientToSteady} 
\end{figure}

After a transient state corresponding to priority values higher than $p_{act}$, the system stabilizes to a periodic state. Priorities oscillate around $p_{act}$ and $\tilde{p}_{max}(t)$ is constant, as expected from the theory.
\end{subsection}
\begin{subsection}{Total time spent in the network}
We now derive the total time spent in the network, assuming that the average queue size is the average of the bounds from Theorem~\ref{theo:Qbounds}.
\begin{definition}
	The average queue size of $\bar{q}_s$ is approximated by:
	\begin{equation*}
		\bar{q}_s = q_{act} + f_s - \frac{1}{2}c_s
		\label{eqn:Qbar}
	\end{equation*}
	\label{def:Comparison}
\end{definition}
\noindent The queue size translates directly into time spent given our implicit choice of a unit time step. We can now compare the classical backpressure algorithm and the proposed backpressure algorithm based on total time spent in the network.
\begin{theorem}
	In regime $R1$, the ratio of total time spent in the network as the heterogeneity increases is:
	\begin{itemize}
		\item with $(u,s) = (1,2)$: $\frac{\bar{q}_{s,classical}}{\bar{q}_{s,proposed}} \sim_{k \rightarrow \infty} 1$
		
		\item with $(u,s) = (2,1)$: $\frac{\bar{q}_{s,classical}}{\bar{q}_{s,proposed}} \sim_{k \rightarrow \infty} k$
	\end{itemize}
	\label{theo:Better}
\end{theorem}
\begin{proof}
We expand the expression of total time spent from Definition~\ref{def:Comparison} using the detailed values from Theorem~\ref{theo:Qbounds}.
\end{proof}
\end{subsection}
\end{section}
\begin{section}{Numerical Results}\label{sec:NumRes}
We first introduce the experimental setup, then go over results of benchmark experiments, and finally analyze performance on scenarios mimicking realistic conditions.
\begin{subsection}{Experimental setup}
We consider a heterogeneous Manhattan grid wherein some links have high capacity (major arterial roads) and other links have low capacity (secondary arterial roads). At each junction, $3$ distinct movements are authorized (left, straight, right) with the straight movement having double capacity.

The average demand for each origin node, destination node pair is drawn from an exponential distribution. For every origin-destination pair, the path minimizing the travel-time at the speed limit is computed, and flow is assigned accordingly. 

Routing rates at each junction for the aggregate flow are computed from the full assignment, by computing the proportion of flow using each movement compared to the incoming flow. Input flow is drawn from a Poisson law with mean defined for each origin-destination pair as explained above.

Performance metrics such as time spent in the network are computed by summing over time the sizes of queues.
\end{subsection}
\begin{subsection}{Benchmark experiments}
In this section we perform controlled experiments to investigate and validate various properties of the proposed algorithm, compared to the classical backpressure algorithm. We consider a $10\times10$ Manhattan grid with a major arterial every 5 blocks, and a time step of $30$ seconds.

First we analyze the impact of the parameter $\rho$ characterizing the magnitude of the demand. Figure~\ref{fig:ExpTraffic} displays the ratio of time spent in the network for $500$ time steps.
\begin{figure}[htb!] 
	\centering
	\includegraphics[width=\columnwidth]{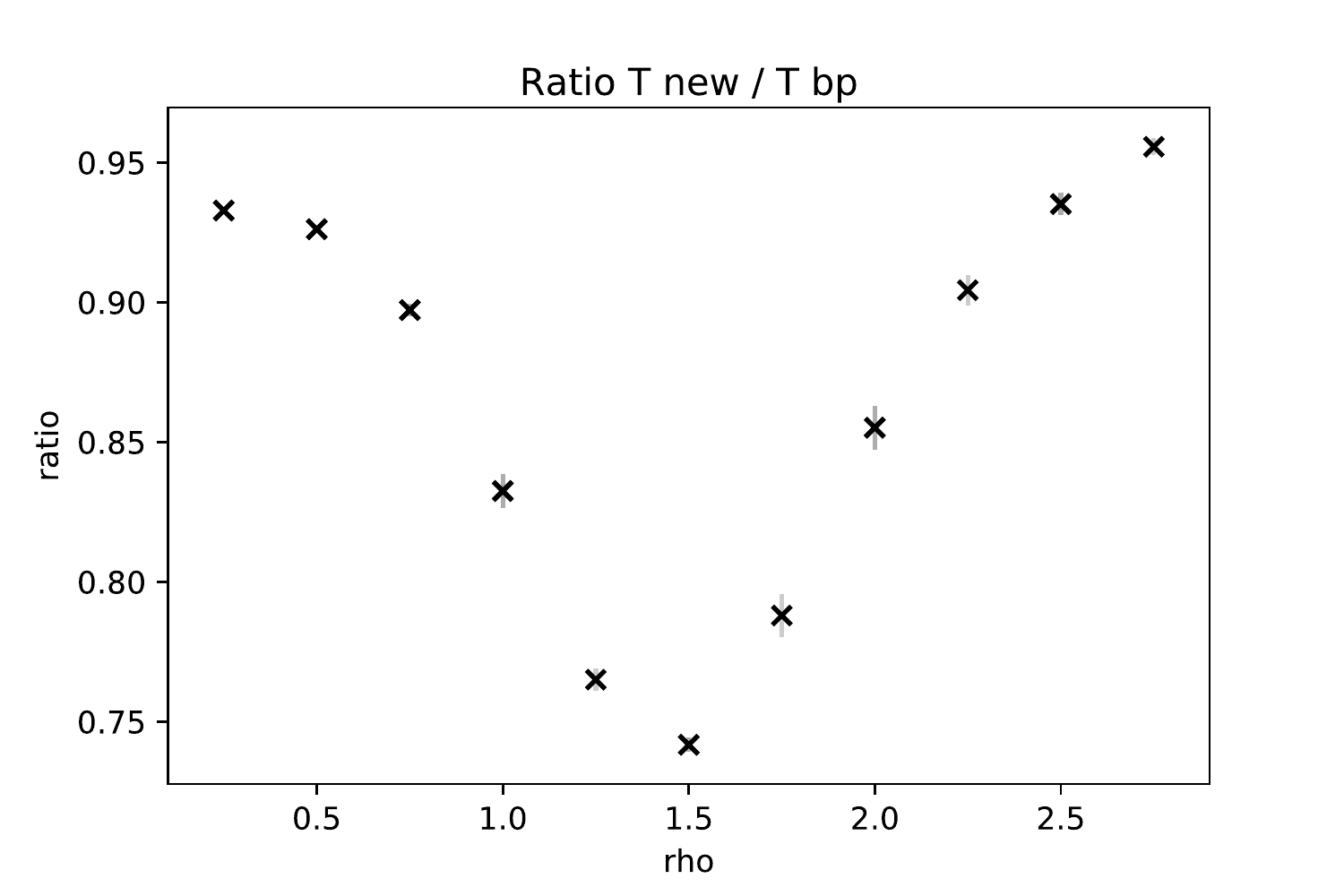}
	\caption{\textbf{Ratio of time spent in the network} for the two algorithms for increasing demand. Every point represents $300$ simulations and displays the deviation of the sample for classical backpressure (``bp'') and proposed algorithm (``new'').}
	\label{fig:ExpTraffic}
\end{figure}

These results confirm the theoretical results on the existence of three regimes depending on the values taken by $k,Q$: for low demand (low $\rho$ values), the flow is not really constrained so both algorithms have similar performance, for demand around and slightly above capacity there is little available supply and the new algorithm improves on the classical backpressure, and for higher demand the network is too saturated to leave room for optimization and both algorithms have similar performance. In this setting, the network is unstable for $\rho > 2$. The improvement is at most $25\%$.

Second we analyze the impact of the parameter $h$ characterizing the network heterogeneity, with $1/h$ being the density of major arterial roads; $h = 0$ corresponds to no major arterial, $h = 1$ corresponds to all roads being major arterial, 2 to one over 2, etc. The results of simulations are presented in Figure~\ref{fig:ExpH_highways}.

\begin{figure}[htb] 
	\centering
	\includegraphics[width=\columnwidth]{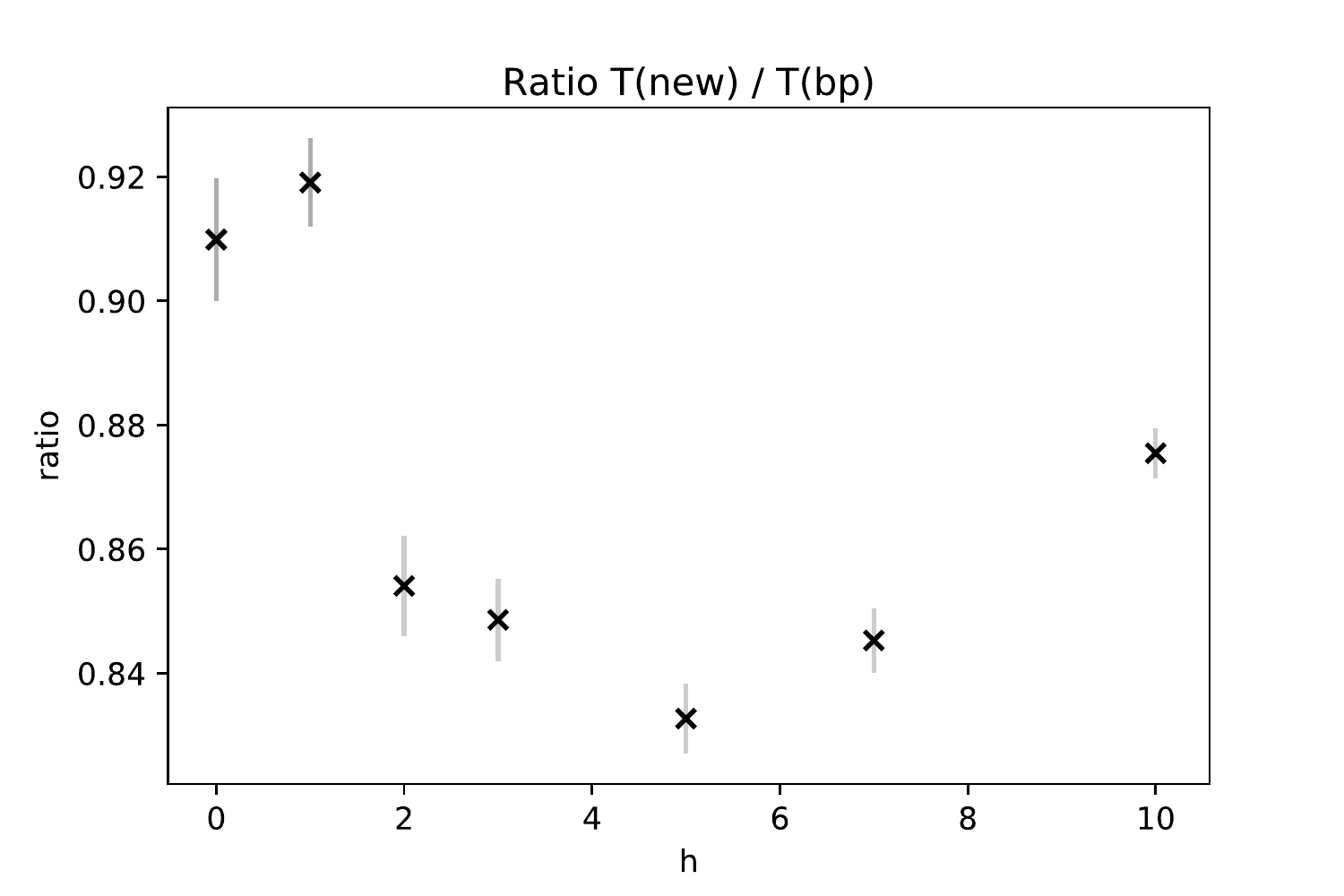}
	\caption{\textbf{Ratio of time spent in the network} for the two algorithms for different distances between parallel highways. Every point represents $300$ simulations, for classical backpressure (``bp'') and proposed algorithm (``new'').}
	\label{fig:ExpH_highways}
\end{figure}

Low values of $h$ correspond to a homogeneous network and the experiments confirm ($h = 0,1,2$) that the proposed algorithm has no significant impact on traffic compared to the classical backpressure. For higher $h$ the network is more heterogeneous and as expected the proposed algorithm has better relative performance.

We now analyze the impact of the ratio of the capacities between major arterial roads and secondary arterial roads. The results of the simulations are displayed in Figure~\ref{fig:ExpFactor}. 
\begin{figure}[htb!] 
	\centering
	\includegraphics[width=\columnwidth]{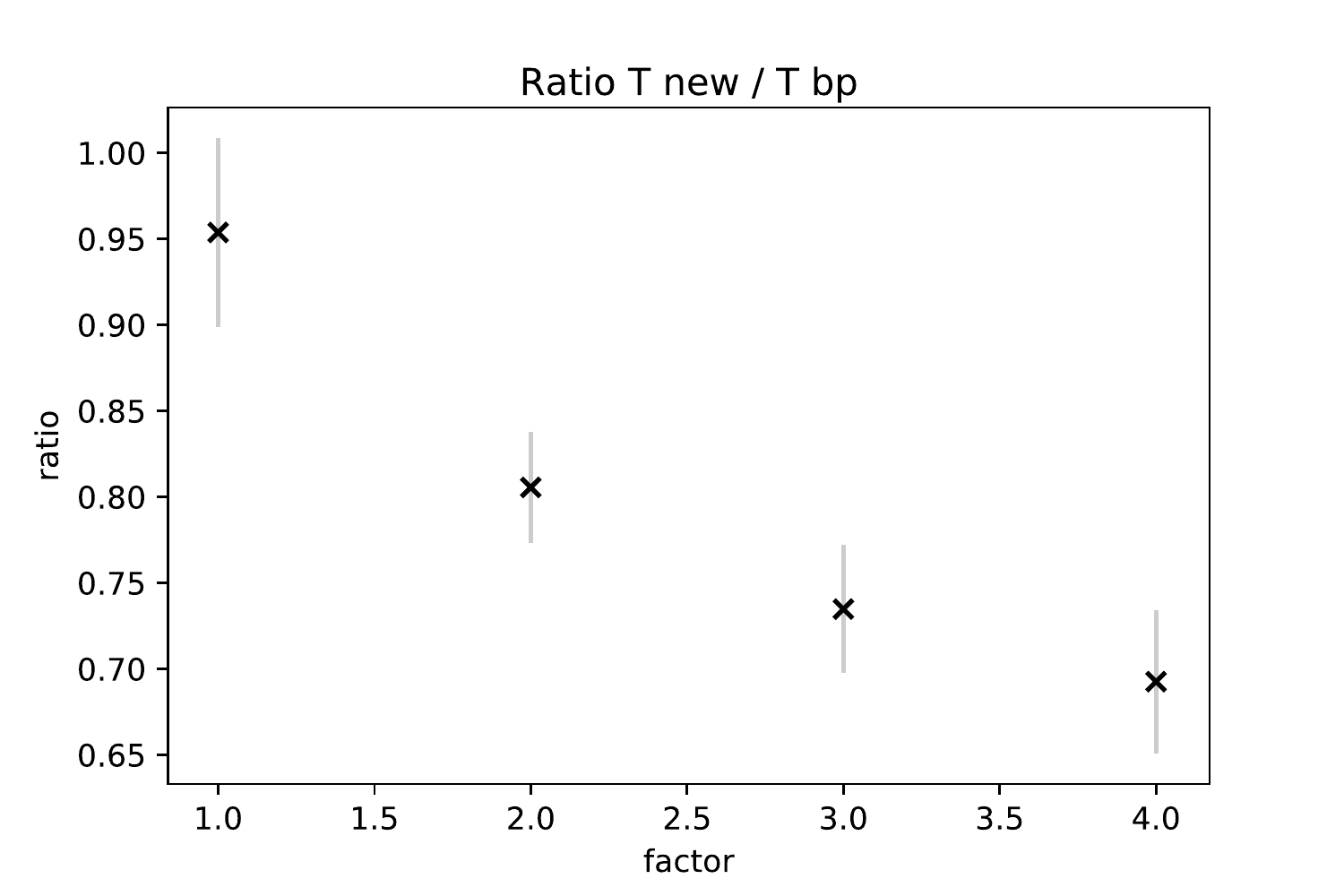}
	\caption{\textbf{Ratio of time spent in the network} for the two algorithms for increasing capacity heterogeneity. Every point represents $300$ simulations, for classical backpressure (``bp'') and proposed algorithm (``new'').}
	\label{fig:ExpFactor}
\end{figure}

For a homogeneous network (capacity ratio close to $1$), the performance of both algorithms is similar. The comparative performance of the new algorithm with respect to the regular backpressure increases as the capacity ratio increases. The higher the capacity ratio the higher the flow on the major arterials (because their attractivity increases) so the more flow heterogeneity there is between major arterial and other roads.
\begin{figure}[htb!] 
	\centering
	\includegraphics[width=\columnwidth]{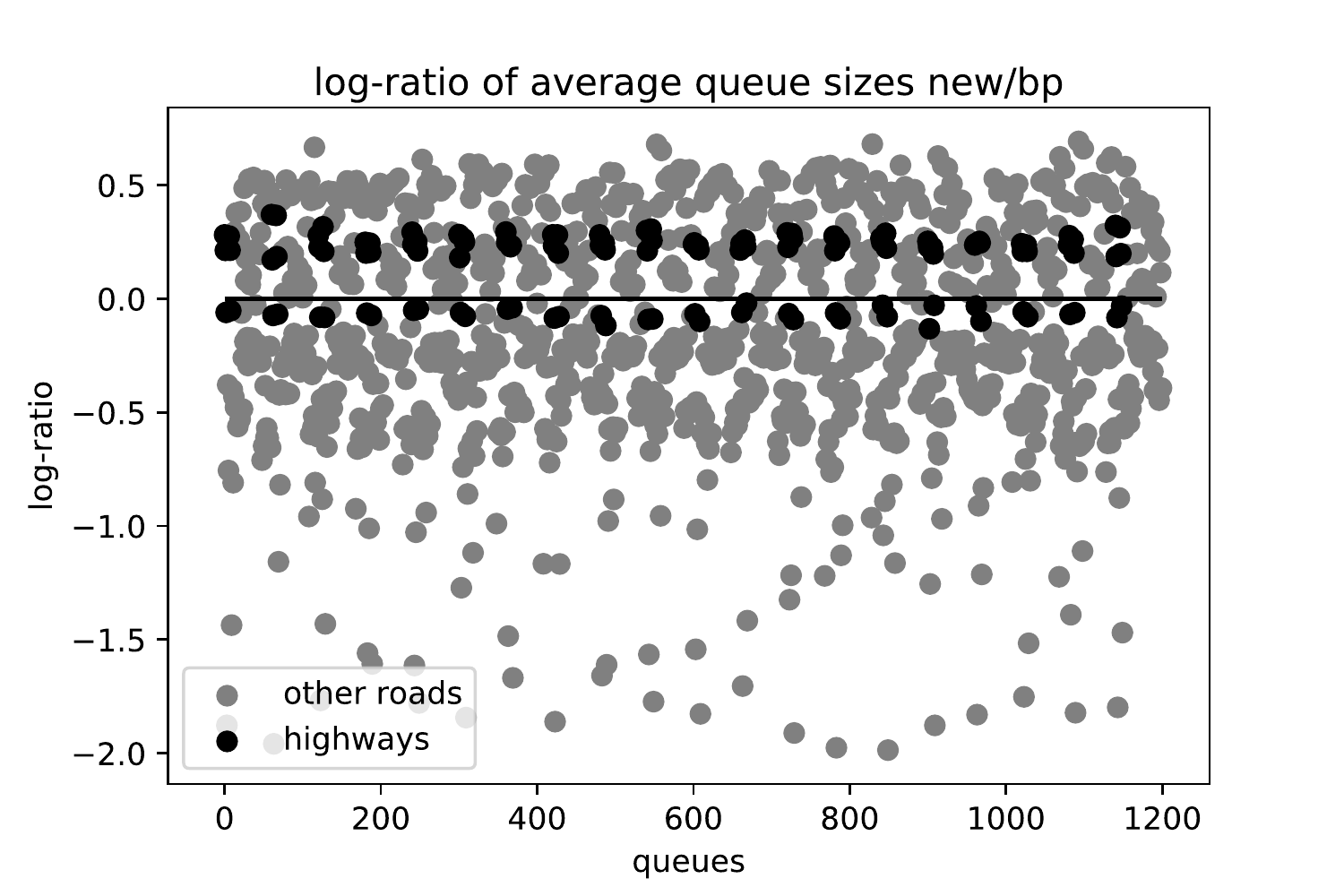}
	\caption{\textbf{Log-ratio of the average queue size} for the proposed algorithm versus the classical algorithm, for every link in the network, including major arterials (highways) and secondary arterials (other roads).}
	\label{fig:ExpScatter}
\end{figure}

These cases correspond to the lower part of Figure~\ref{fig:BothMono} for the single junction case where the proposed algorithm was proven theoretically to have better performance. In order to understand how the proposed algorithm achieves better performances we compare the average size of each queue for a heterogeneous network in Figure~\ref{fig:ExpScatter}. 

The results highlight that the proposed algorithm results in a moderate queue size increase on major arterials but yields significant queue size reduction on secondary arterial roads, meaning that the proposed algorithm is better able to take advantage of the network heterogeneity. In the next section we further explore the algorithm performance in realistic scenarios.
\end{subsection}
\begin{subsection}{Peak hour scenario}
In this section, we analyze a scenario modeled after a peak hour. The specific network considered is a Manhattan network with a $50\times10$ grid with a major arterial road every 4 blocks. The factor $\rho$ varies over time in a triangular shape from $0$ to $3$ to model a demand temporally exceeding the network capacity ($\rho > 2$ corresponds to an instable network).
\begin{figure}[htb!] 
	\centering
	\includegraphics[width=\columnwidth]{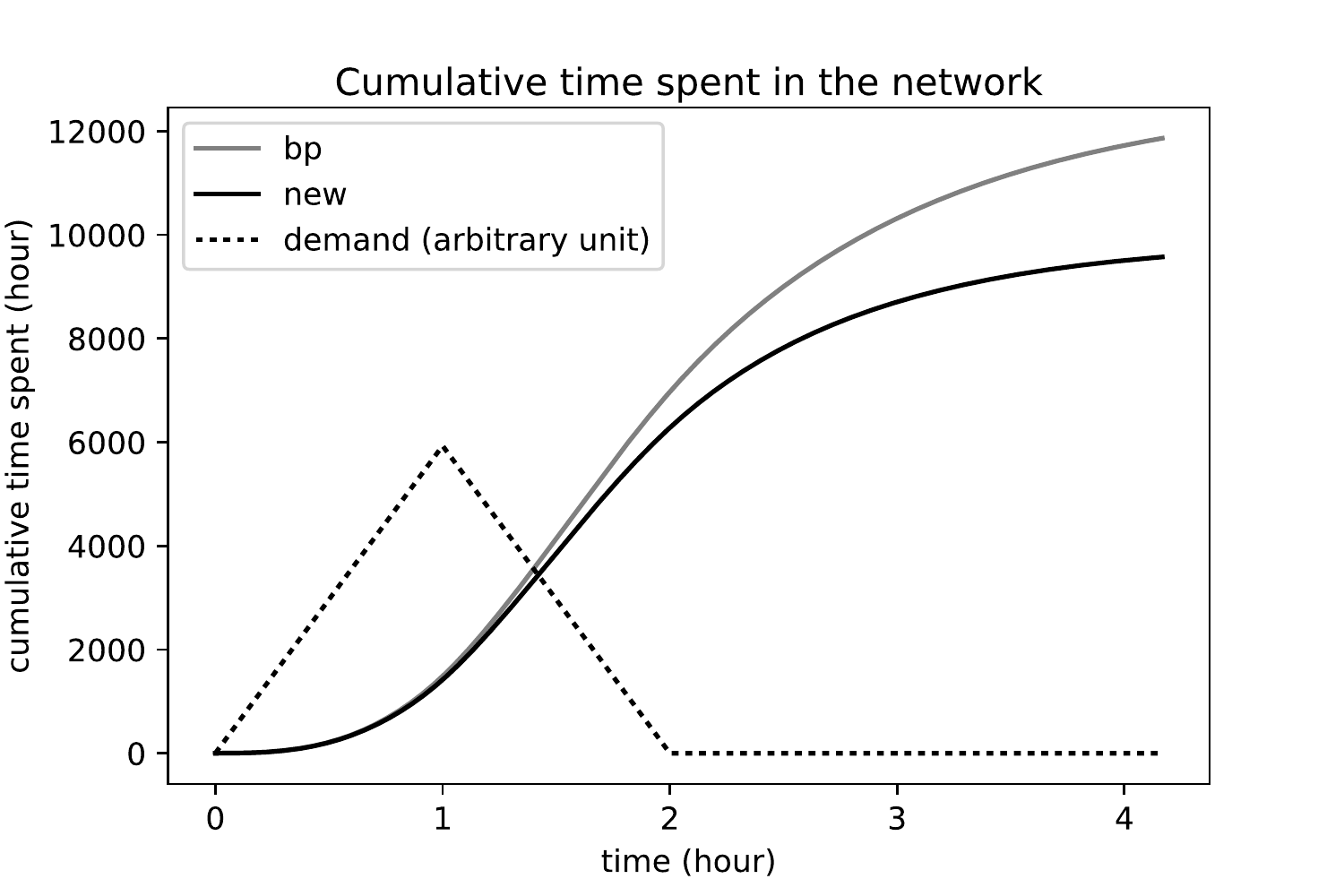}
	\caption{\textbf{Cumulative time spent in the network} for both algorithms (the unscaled value of $\rho$ is also represented).}
	\label{fig:ExpPeak}
\end{figure}

Figure~\ref{fig:ExpPeak} illustrates that the excess demand causes the cumulative time spent in the network to increase in all cases. However the increase is less important in the case of the proposed algorithm, in particular when the demand returns to a low value, and subsequently. This can be explained by the fact that in the first part of the peak period, the network is capacity-constrained, hence no algorithm is given sufficient freedom to optimize. However in the second phase of the peak time, and subsequently, the proposed algorithm is able to take better advantage of available capacity.
\end{subsection}
\begin{subsection}{Incident scenario}
The second scenario that we consider is an incident modeled as a link with zero capacity for a one hour period in a Manhattan network. The network size and simulation parameters are chosen such that the boundary links are never impacted, and the incident happens after the network loading period, and clears before the end of the simulation. 

We first analyze in Figure~\ref{fig:IncidentQueues} the extent to which queues are smoothed in space via the backpressure effect of preserving already large queues from further inflow.
\begin{figure}[htb!] 
	\centering
	\includegraphics[width=\columnwidth]{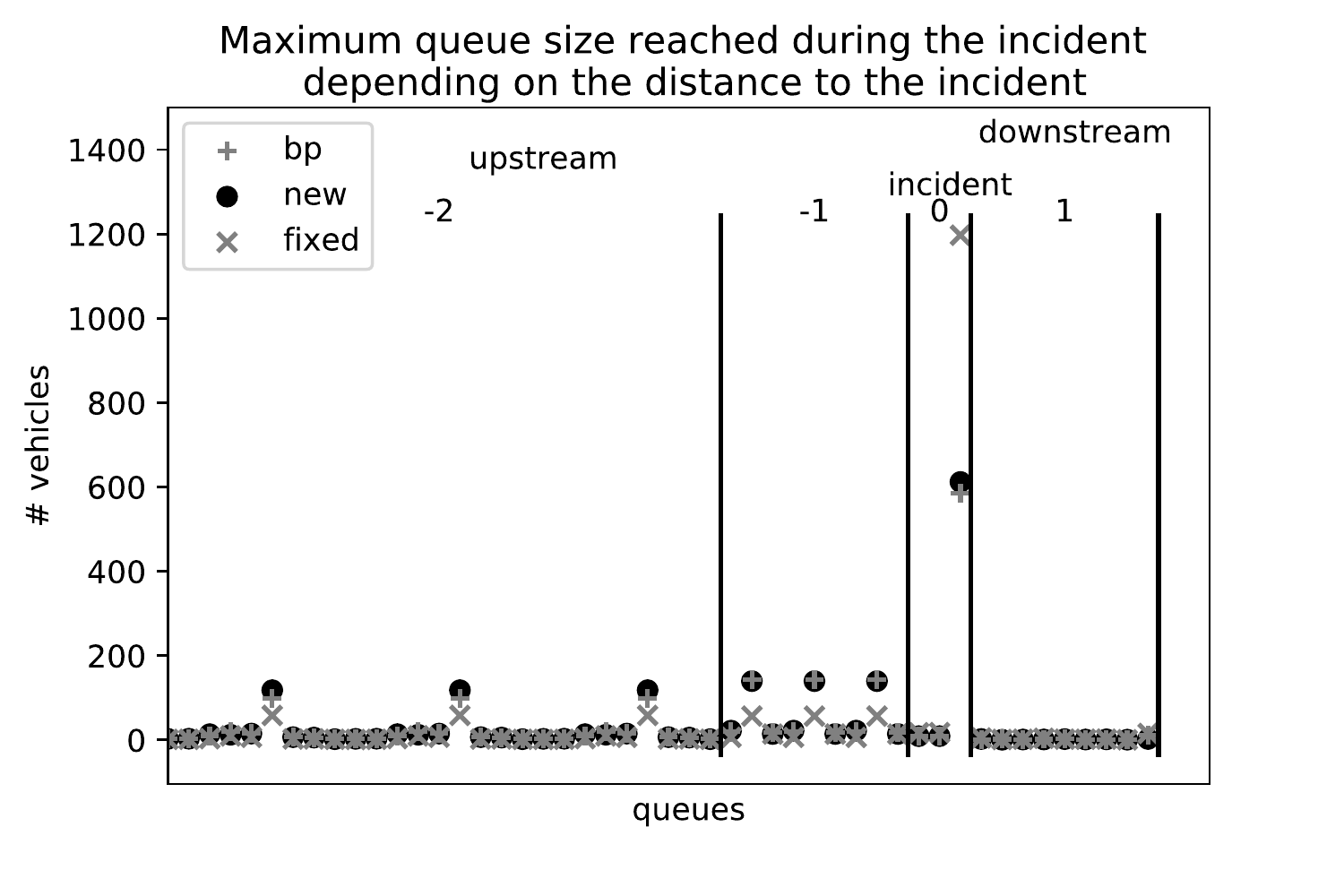}
	\caption{\textbf{Maximum queue size:} as a function of distance to incident (at the intersection directly connected to the incident link, 1 hop and 2 hops upstream, as well as 1 hop downstream) for the classical backpressure (``bp''), the proposed algorithm (``new''), and the fixed cycle policy (``fixed'').}
	\label{fig:IncidentQueues}
\end{figure}

For both the classical backpressure and the proposed algorithm, the queues at the incident location are much lower compared to the fixed cycle policy, which in practice reduces the chances of grid-lock. This is achieved at the cost of having slightly longer queues upstream of the incident.

We now consider the cumulative queue length across the network as a function of time in Figure~\ref{fig:Incident}.
\begin{figure}[htb!] 
	\centering
	\includegraphics[width=\columnwidth]{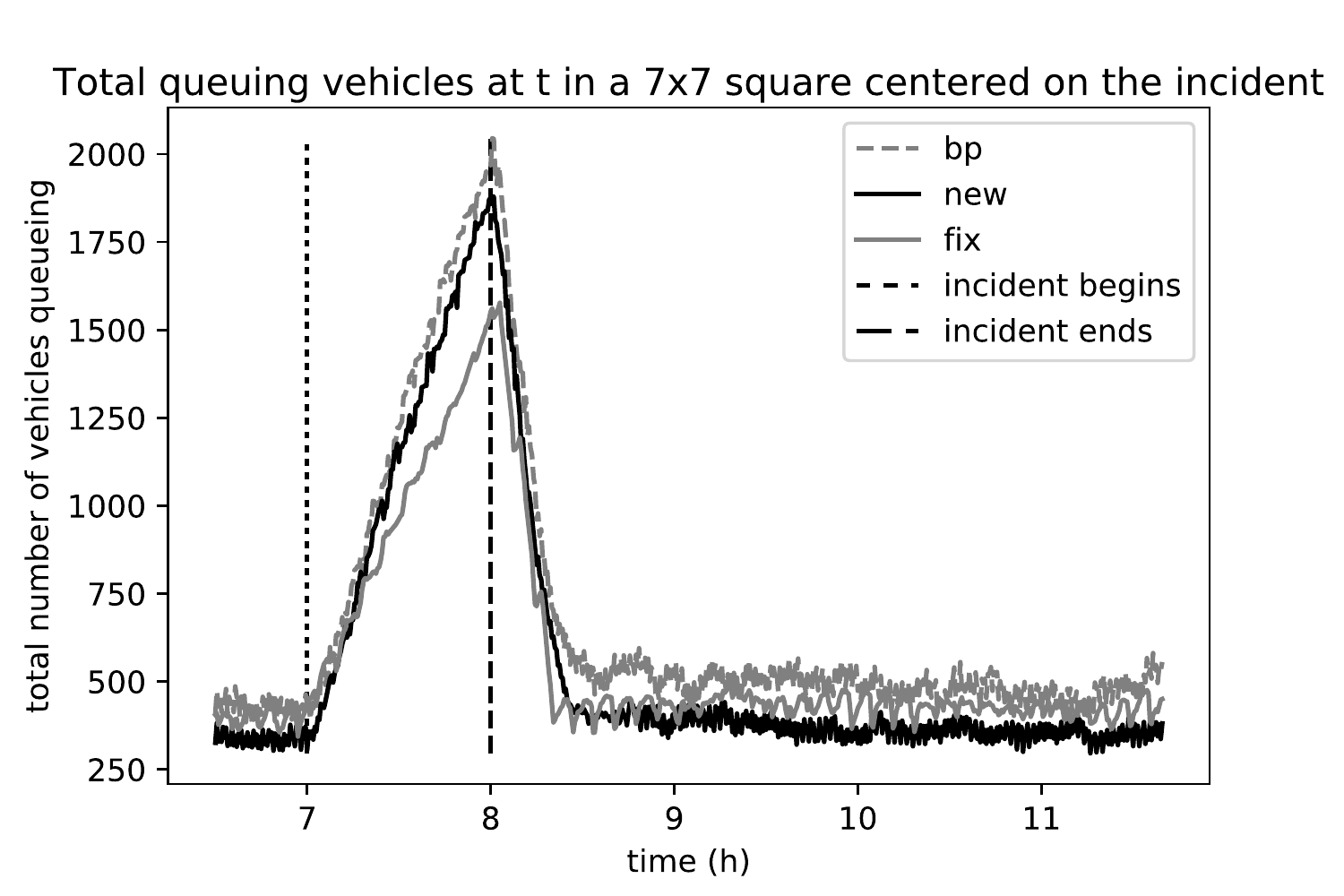}
	\caption{\textbf{Queue size in the vicinity of the incident link}: under the classical backpressure policy (``bp''), the proposed control policy (``new''), and a fixed cycle policy (``fix'') as a benchmark. The incident start and end times are indicated with vertical dashed lines.}
	\label{fig:Incident}
\end{figure}

We observe that both the classical backpressure and the proposed algorithm outperform the fixed cycle policy, and the proposed algorithm slightly improves on the classical backpressure, with the difference between all methods increasing as the network gets more saturated, since the key benefit of the adaptive scheduling algorithms lies in efficiently using available link capacities around saturated conditions.
\end{subsection}
\end{section}
\begin{section}{Conclusion}\label{sec:Concl}
In this work we investigated the problem of the convergence properties of an approximate gradient descent method for network adaptive control. 

Using a fundamental $2\times1$ network, we proved that different regimes exist depending on exogenous parameters such as the magnitude of the demand with respect to the network capacity, and the heterogeneity of the flow across competing links. We characterized each of these regimes theoretically, and verified in simulation on realistic network the expected theoretical properties. 

As part of this analysis, we also showed that appropriate calibration of the weights in the objective function can significantly improve the asymptotic objective value, by up to $\mathcal{O}(k)$ with $k$ the ratio of competing queues capacity.
\end{section}
\bibliography{MAIN-CarpentierJ.272.bib}

\begin{thebibliography}{10}

\bibitem{faster2019}
Sebastien Blandin, Laura Wynter, Hasan Poonawala, Sean Laguna, and Basile Dura.
\newblock {FASTER: Fusion AnalyticS for public Transport Event Response}.
\newblock In {\em Proceedings of the 18\textsuperscript{th} Conference on
  Autonomous Agents and MultiAgent Systems}. International Foundation for
  Autonomous Agents and Multiagent Systems, 2019.

\bibitem{Boyer2015}
Sebastien Boyer, Sebastien Blandin, and Laura Wynter.
\newblock {Stability of Transportation Networks Under Adaptive Routing
  Policies}.
\newblock {\em Transportation Research Procedia}, 7:578--597, 2015.

\bibitem{Bramson99}
Maury Bramson.
\newblock Stability and heavy traffic limits for queueing networks.
\newblock In {\em Perplexing Problems in Probability}, pages 247--267.
  Springer, 1999.

\bibitem{Blandin2015}
Charles Brett, Saif Jabari, Sebastien Blandin, and Laura Wynter.
\newblock In {\em 95th Annual Meeting of the Transportation Research Board,
  Washington, D.C.}, 2016.

\bibitem{Stolyar09}
Loc Bui, R~Srikant, and Alexander Stolyar.
\newblock Novel architectures and algorithms for delay reduction in
  back-pressure scheduling and routing.
\newblock In {\em INFOCOM 2009, IEEE}, pages 2936--2940, 2009.

\bibitem{daganzo1995ctm}
C.~Daganzo.
\newblock The cell transmission model, part \text{II}: Network traffic.
\newblock {\em Transportation Research Part B: Methodological}, 29(2):79--93,
  1995.

\bibitem{piccoli}
M.~Garavello and B.~Piccoli.
\newblock {\em Traffic flow on networks}.
\newblock American Institute of Mathematical Sciences, Springfield, MO, 2006.

\bibitem{gregoire2014back}
Jean Gregoire, Emilio Frazzoli, Arnaud De~La~Fortelle, and Tichakorn
  Wongpiromsarn.
\newblock Back-pressure traffic signal control with unknown routing rates.
\newblock {\em IFAC Proceedings Volumes}, 47(3):11332--11337, 2014.

\bibitem{gregoire2015capacity}
Jean Gregoire, Xiangjun Qian, Emilio Frazzoli, Arnaud De~La~Fortelle, and
  Tichakorn Wongpiromsarn.
\newblock Capacity-aware backpressure traffic signal control.
\newblock {\em IEEE Transactions on Control of Network Systems}, 2(2):164--173,
  2015.

\bibitem{horni2016multi}
Andreas Horni, Kai Nagel, and Kay~W Axhausen.
\newblock {\em The multi-agent transport simulation MATSim}.
\newblock Ubiquity Press London:, 2016.

\bibitem{hucoping}
Hsu-Chieh Hu and Stephen~F Smith.
\newblock Coping with large traffic volumes in schedule-driven traffic signal
  control.
\newblock In {\em Proceedings of the 27th International Conference on Automated
  Planning and Scheduling}, pages 154--162, 2017.

\bibitem{Hu2017a}
Hsu-Chieh Hu and Stephen~F Smith.
\newblock Softpressure: a schedule-driven backpressure algorithm for coping
  with network congestion.
\newblock In {\em Proceedings of the 26th International Joint Conference on
  Artificial Intelligence}, pages 4324--4330, 2017.

\bibitem{Moeller10}
Scott Moeller, Avinash Sridharan, Bhaskar Krishnamachari, and Omprakash
  Gnawali.
\newblock Routing without routes: The backpressure collection protocol.
\newblock In {\em Proceedings of the 9\textsuperscript{th} ACM/IEEE
  International Conference on Information Processing in Sensor Networks}, pages
  279--290. ACM, 2010.

\bibitem{moharir2013maxweight}
Sharayu Moharir and Sanjay Shakkottai.
\newblock Maxweight vs. backpressure: Routing and scheduling in multi-channel
  relay networks.
\newblock In {\em INFOCOM, 2013 Proceedings IEEE}, pages 1537--1545. IEEE,
  2013.

\bibitem{nikolova2006optimal}
Evdokia Nikolova, Matthew Brand, and David~R Karger.
\newblock Optimal route planning under uncertainty.
\newblock In {\em Proceedings of the 16th International Conference on Automated
  Planning and Scheduling}, volume~6, pages 131--141, 2006.

\bibitem{richter2007natural}
Silvia Richter, Douglas Aberdeen, and Jin Yu.
\newblock Natural actor-critic for road traffic optimisation.
\newblock In {\em Advances in Neural Information Processing Systems (NIPS)},
  pages 1169--1176, 2007.

\bibitem{singh2015maxweight}
Rahul Singh and Alexander Stolyar.
\newblock Maxweight scheduling: Asymptotic behavior of unscaled
  queue-differentials in heavy traffic.
\newblock In {\em ACM SIGMETRICS Performance Evaluation Review}, volume~43,
  pages 431--432, 2015.

\bibitem{smith2013smart}
Stephen~F Smith, Gregory~J Barlow, Xiao-Feng Xie, and Zachary~B Rubinstein.
\newblock Smart urban signal networks: Initial application of the surtrac
  adaptive traffic signal control system.
\newblock In {\em Proceedings of the 13th International Conference on Automated
  Planning and Scheduling}, 2013.

\bibitem{Stolyar11}
Alexander~L Stolyar.
\newblock Large number of queues in tandem: Scaling properties under
  back-pressure algorithm.
\newblock {\em Queueing Systems}, 67(2):111--126, 2011.

\bibitem{Stolyar04}
Alexander~L Stolyar et~al.
\newblock Maxweight scheduling in a generalized switch: State space collapse
  and workload minimization in heavy traffic.
\newblock {\em The Annals of Applied Probability}, 14(1):1--53, 2004.

\bibitem{tassiulas1992stability}
Leandros Tassiulas and Anthony Ephremides.
\newblock Stability properties of constrained queueing systems and scheduling
  policies for maximum throughput in multihop radio networks.
\newblock {\em IEEE Transactions on Automatic Control}, 37(12):1936--1948,
  1992.

\bibitem{tenbusch2014guaranteeing}
Simon Tenbusch, Christof L{\"o}ding, Frank Radmacher, and James Gross.
\newblock Guaranteeing stability and delay in dynamic networks based on
  infinite games.
\newblock In {\em 11th International Conference on Mobile Ad Hoc and Sensor
  Systems (MASS)}, pages 461--469, 2014.

\bibitem{varaiya2009universal}
P~Varaiya.
\newblock A universal feedback control policy for arbitrary networks of
  signalized intersections.
\newblock {\em Published online. URL: http://paleale. eecs. berkeley.
  edu/varaiya/papers$\backslash$ps. dir/090801-Intersectionsv5. pdf}, 2009.

\bibitem{Venkataramanan10}
VJ~Venkataramanan, Xiaojun Lin, Lei Ying, and Sanjay Shakkottai.
\newblock On scheduling for minimizing end-to-end buffer usage over multihop
  wireless networks.
\newblock In {\em INFOCOM, 2010 Proceedings IEEE}, pages 1--9, 2010.

\bibitem{Wongpiromsarn2014}
Tichakorn Wongpiromsarn.
\newblock {Throughput Optimal Distributed Traffic Signal Control}.
\newblock {\em arXiv preprint arXiv: \ldots}, pages 1--23, 2014.

\bibitem{Wunder12}
Gerhard Wunder and Martin Kasparick.
\newblock Universal stability and cost optimization in controlled queueing
  networks.
\newblock In {\em IEEE Wireless Communications and Networking Conference
  (WCNC)}, pages 3069--3073, 2012.

\bibitem{xie2012schedule}
Xiao-Feng Xie, Stephen~F Smith, and Gregory~J Barlow.
\newblock Schedule-driven coordination for real-time traffic network control.
\newblock In {\em Proceedings of the 12th International Conference on Automated
  Planning and Scheduling}, 2012.

\bibitem{Ying11}
Lei Ying, Sanjay Shakkottai, Aneesh Reddy, and Shihuan Liu.
\newblock On combining shortest-path and back-pressure routing over multihop
  wireless networks.
\newblock {\em IEEE/ACM Transactions on Networking (TON)}, 19(3):841--854,
  2011.

\bibitem{Zargham13}
Michael Zargham, Alejandro Ribeiro, and Ali Jadbabaie.
\newblock Accelerated backpressure algorithm.
\newblock In {\em IEEE Global Communications Conference (GLOBECOM)}, pages
  2269--2275, 2013.

\end{thebibliography}
\bibliographystyle{plain}
\end{document}